\documentclass[a4paper,12pt]{article}
\usepackage[top=2.5cm,bottom=2.5cm,left=2.5cm,right=2.5cm]{geometry}
\usepackage{cite,amsmath,amssymb}
    \usepackage[margin=1cm,%
                font=small,%
                format=hang,%
                labelsep=period,%
                labelfont=bf]{caption}
\usepackage[algoruled,linesnumbered]{algorithm2e}
\usepackage{algorithmic}

\usepackage{amsfonts,epsf,here,graphicx}
\usepackage{latexsym,multirow,rotating}
\usepackage{pgf,tikz,subfigure}
\usepackage{epstopdf}

\newtheorem{theorem}{\bf Theorem}[section]

\newtheorem{lemma}[theorem]{\bf Lemma}

\newcommand{\qed}{\hfill $\square$ \bigskip}

\begin{document}

\baselineskip=0.30in
\vspace*{40mm}

\begin{center}
{\LARGE \bf The Edge-Wiener Index, the Szeged Indices and the PI Index of Benzenoid Systems in Sub-Linear Time}
\bigskip \bigskip

{\large \bf Matev\v z \v Crepnjak$^{a,b,c}$ \qquad Niko Tratnik$^a$
}
\bigskip\bigskip

\baselineskip=0.20in

$^a$ \textit{Faculty of Natural Sciences and Mathematics, University of Maribor, Slovenia} \\
{\tt matevz.crepnjak@um.si, niko.tratnik@um.si}
\medskip

$^b$ \textit{Faculty of Chemistry and Chemical Engineering, University of Maribor, Slovenia}
\medskip

$^c$ \textit{Andrej Maru\v si\v c Institute, University of Primorska, Slovenia}

\bigskip\medskip

(Received \today)

\end{center}

\noindent
\begin{center} {\bf Abstract} \end{center}

\vspace{3mm}\noindent
In this paper, we investigate the edge-Wiener index, the Szeged index, the edge-Szeged index, and the PI index, which are some of the most studied distance-based topological indices. As the main result we show that for benzenoid systems these indices can be computed in sub-linear time with respect to the number of vertices. More precisely, they can be computed in the time dependent on the length of the boundary cycle of a benzenoid system.
\baselineskip=0.30in



\section{Introduction}
\label{sec:intro}
Molecular structure-descriptors, also called topological indices, are used in theoretical chemistry for the design of quantitative structure-property relations (QSPR) and quantitative structure-activity relations (QSAR). Distance-based topological indices are defined by using distances in molecular graphs and can be very useful in a drug discovery process, see \cite{madan}.

The most famous distance-based topological index is the Wiener index and it was first introduced in 1947 by H. Wiener \cite{wiener}. The Wiener index of a connected graph $G$ is defined as

$$W(G) = \sum_{\lbrace u,v \rbrace \subseteq V(G)} d_G(u,v).$$

\noindent It has known correlations with a large number of physico-chemical properties of organic molecules and also possesses interesting mathematical properties. Therefore, the Wiener index has been extensively studied in mathematical and chemical literature, see \cite{knor2}.

It turns out that the Wiener index of a tree can be computed as the sum of edge contributions. Inspired by this fact, the Szeged index was introduced (see \cite{gut_sz}) as
$$Sz(G) = \sum_{e=uv \in E(G)}n_u(e)n_v(e),$$
where $n_u(e)$ denotes the number of vertices of a graph $G$ whose distance to $u$ is smaller than the distance to $v$ and $n_v(e)$ denotes the number of vertices of $G$ whose distance to $v$ is smaller than the distance to $u$. Motivated by the success of the Szeged index, in \cite{def_pi} a similar molecular descriptor called the PI index (or the edge-PI index) was defined by
$$PI(G) = \sum_{e=uv \in E(G)}\big(m_u(e) + m_v(e)\big),$$
where the numbers $m_u(e)$ and $m_v(e)$ are the edge-variants
of the numbers $n_u(e)$ and $n_v(e)$. Later \cite{khal1}, a vertex version of the PI index, called the vertex-PI index, was also defined as $PI_v(G) = \sum_{e=uv \in E(G)}\big(n_u(e) + n_v(e)\big)$. Obviously, for any bipartite graph $G$ the vertex-PI index can be computed as $PI_v(G) = |V(G)| \cdot |E(G)|$. In 2008, the edge-version of the Szeged index, called the edge-Szeged index, was defined (see \cite{gut}) as
$$Sz_e(G) = \sum_{e=uv \in E(G)}m_u(e)m_v(e).$$
Papers \cite{al-fozan,chen1,koor,tratnik-zig,wang0,wang} present a sample of relevant recent investigations on the mentioned distance-based topological indices.

Finally, the edge-Wiener index of a graph was independently introduced in~\cite{iranmanesh-2009,khalifeh-2009}. In~\cite{iranmanesh-2009} several possible variations of the concept were discussed and it was suggested that the edge-Wiener index of a graph $G$ should be defined as the Wiener index of the line graph of $G$. For some recent studies on the edge-Wiener index see \cite{aro,chen2,knor}.

In this paper, we investigate the mentioned indices of benzenoid systems, which are one of the most extensively studied family of chemical graphs. In \cite{chepoi-1997} it was proved that the Wiener index and the Szeged index of a benzenoid system can be computed in linear time. Almost 20 years later, in \cite{kelenc} an algorithm was developed that, for a given benzenoid system $G$ computes the edge-Wiener index of $G$ in linear time. Moreover, similar linear time algorithms for the edge-Szeged index and the PI index of benzenoid systems were developed in \cite{tratnik}. All these algorithms are based on the cut method and more information about this method can be found in survey papers \cite{klavzar-2008,klavzar-2015}. 

In addition to this, Chepoi and Klav\v zar proved \cite{CK-1998} that the Wiener index of a benzenoid system can be computed in the time dependent on the number of vertices in the boundary cycle of a benzenoid system. More precisely, the Wiener index of a benzenoid system with the boundary cycle $Z$ can be computed in $O(|Z|)$ time, where $|Z|$ denotes the number of vertices on cycle $Z$. The proof of this result is based on the result of Chazelle \cite{chazelle}, where he develop an algorithm for computing all vertex-edge visible pairs of edges of a simple (finite) polygon. 

Therefore, we generalize the mentioned results and prove that the edge-Wiener index, the Szeged index, the edge-Szeged index, and the PI index  of a benzenoid system with the boundary cycle $Z$ can be computed in $O(|Z|)$ time.

We proceed as follows. In the next section we give some definitions and important concepts needed later. In section 3 we prove that the Szeged index of a benzenoid system can be computed in sub-linear time. To prove this, we recall the result claiming that the Szeged index of a benzenoid system can be expressed as the sum of weighted Szeged indices of related weighted quotient trees. A method to obtain the weighted trees is also presented. Similar results for the edge-Wiener index, the edge-Szeged index, and the PI index are stated in section 4, but in these cases our method to obtain the weighted quotient trees is different and requires some additional insights.

\section{Preliminaries}

In the present paper all graphs are simple, finite and connected. We define $d_G(u,v)$ to be the usual shortest-path distance between two vertices $u, v\in V(G)$. In addition, for a vertex $x \in V(G)$ and an edge $e=uv \in E(G)$ we set 
\begin{equation*}
d_G(x,e) = \min \lbrace d_G(x,u), d_G(x, v) \rbrace\,.
\end{equation*}

\noindent
The distance $d_G(e,f)$ between edges $e$ and $f$ of a graph $G$ is defined as the distance between vertices $e$ and $f$ in the line graph $L(G)$. Here we follow this convention because in this way the pair $(E(G),d_G)$ forms a metric space.
\bigskip

\noindent
The {\em edge-Wiener index} of a graph $G$ is defined as
\begin{equation} \label{def_ew}
W_e(G) = \frac{1}{2} \sum_{e \in E(G)} \sum_{f \in E(G)} d_G(e,f).
\end{equation}

\noindent Let $G$ be a graph and $e=uv$ an edge of $G$. Throughout the paper we will use the following notation:
$$N_1(e|G) = \lbrace x \in V(G) \ | \ d_G(x,u) < d_G(x,v) \rbrace, $$
$$N_2(e|G) = \lbrace x \in V(G) \ | \ d_G(x,v) < d_G(x,u) \rbrace, $$
$$M_1(e|G) = \lbrace f \in E(G) \ | \ d_G(f,u) < d_G(f,v) \rbrace, $$
$$M_2(e|G) = \lbrace f \in E(G) \ | \ d_G(f,v) < d_G(f,u) \rbrace. $$

\noindent
Using introduced notation, the {\em Szeged index} of a graph $G$ is defined as 
$$Sz(G) =\sum_{e \in E(G)}|N_1(e|G)| \cdot |N_2(e|G)|.$$

\noindent
The {\em edge-Szeged} index is defined by the formula
$$Sz_e(G) =\sum_{e \in E(G)}|M_1(e|G)| \cdot |M_2(e|G)|.$$

\noindent
The \textit{PI index} (or the \textit{edge-PI index}) of a graph $G$ is defined as

$$PI(G) = PI_e(G)  =  \sum_{e \in E(G)}\big(|M_1(e|G)| + |M_2(e|G)| \big). $$
\smallskip

\noindent
Let $G$ be a graph and let $w:V(G)\rightarrow {\mathbb R}^+$  and $w':E(G)\rightarrow {\mathbb R}^+$ be given functions. Then $(G,w)$, $(G,w')$, and $(G,w,w')$ are a {\em vertex-weighted graph}, an {\em edge-weighted graph}, and a {\em vertex-edge weighted graph}, respectively.
\bigskip

\noindent
Let ${\cal H}$ be the hexagonal (graphite) lattice and let $Z$ be a cycle on it. Then a {\em benzenoid system} is induced by the vertices and edges of ${\cal H}$, lying on $Z$ and in its interior. For an example of a benzenoid system see Figure \ref{fig:prim}. Note that in paper \cite{DGKZ-2002} such benzenoid systems are called simple.

\begin{figure}[h!] 
\begin{center}
\includegraphics[scale=0.9]{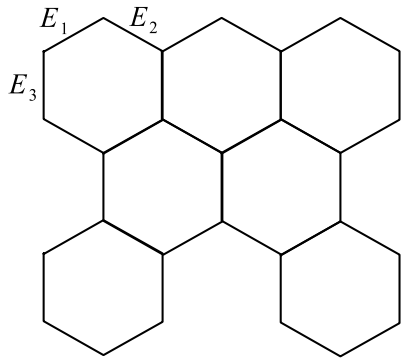}
\end{center}
\caption{\label{fig:prim} A benzenoid system with three directions of edges.}
\end{figure}
\bigskip

\noindent
The edge set of a benzenoid system $G$ can be naturally partitioned into sets $E_1, E_2$, and $E_3$ of edges of the same direction. For $i \in \lbrace 1, 2, 3 \rbrace$, set $G_i = G - E_i$. Then the connected components of the graph $G_i$ are paths. The \textit{quotient graph} $T_i$, $i \in \{1,2,3\}$, has these paths as vertices, two such paths (i.e. components of $G_i$) $P_1$ and $P_2$ being adjacent in $T_i$ if some edge in $E_i$ joins a vertex of $P_1$ to a vertex of $P_2$. It is known that $T_1$, $T_2$, and $T_3$ are trees (see \cite{chepoi-1996}).
\bigskip

%

\section{The Szeged index}
\label{sec:szeged_index}
In this section, we prove that the Szeged index of a benzenoid system $G$ can be computed in sub-linear time. 
To prove this, we extend the main idea from papers \cite{CK-1998,DGKZ-2002} where it was proved that the Wiener index of a benzenoid system can be computed 
in sub-linear time with respect to the number of vertices. First we define the weighted quotient trees.

Let $T_1$, $T_2$, $T_3$ be the quotient trees defined in the preliminaries. 
We can extend the quotient trees to weighted trees $(T_i, w_i, w'_i)$, $i \in \{1,2,3\}$, as follows: 
\begin{itemize}
\item for $C \in V(T_i)$, let $w_i(C)$ be the number of vertices in the component $C$ of $G_i = G - E_i$;
\item for $E = C_1C_2 \in E(T_i)$, let $w_i'(E)$ be the number of edges between components $C_1$ and $C_2$.
\end{itemize}

In the following lemma we prove that the weighted trees can be obtained in $O(|Z|)$ time where $Z$ is the boundary cycle of a benzenoid system and $|Z|$ denotes the number of vertices on $Z$.

\begin{lemma}\label{tree_O(z)}
	Let $G$ be a benzenoid system and let $Z$ be its boundary cycle. Then each tree $(T_i, w_i, w_i')$, $i \in \{1,2,3\}$, can be obtained in $O(|Z|)$ time.	
\end{lemma}
\begin{proof}
Let $i \in \{1,2,3\}$. We will describe the construction of  vertex-edge weighted tree $(T_i, w_i, w_i')$ that depends only on the boundary cycle $Z$.
	
	\noindent
	The main idea is based on Chazelle algorithm \cite{chazelle} for computing all vertex-edge visible pairs of edges of a simple (finite) polygon. More precisely, we apply the algorithm of Chazelle for the direction $E_i$ as follows. Let ${\cal D}$ be the region in the plane bounded by cycle $Z$ such that $Z$ is included in $\cal{D}$. We define a \textit{cut segment} as a straight line segment lying completely in ${\cal D}$ and connecting two distinct vertices of $Z$. A \textit{cut segment of type $i$} is a cut segment perpendicular to the edges from the set $E_i$. We divide ${\cal D}$ into strips by all the cut segments of type $i$, see an example in Figure \ref{fig:sub}. Note that some strips can be triangles. Such a subdivision of ${\cal D}$ will be denoted by ${\cal D}_i$. Moreover, any strip of ${\cal D}_i$ can take two values: $1$ and $\frac{1}{2}$. The strip takes value $1$ if it is a rectangle and value $\frac{1}{2}$ otherwise. 
	
	\begin{figure}[h!] 
\begin{center}
\includegraphics[scale=0.8]{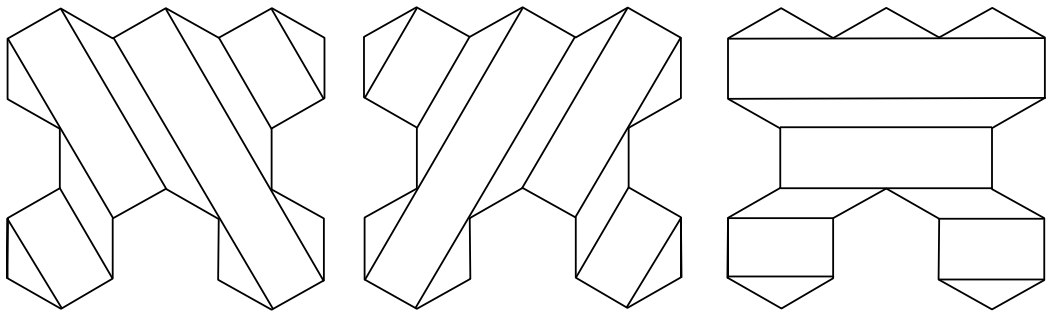}
\end{center}
\caption{\label{fig:sub} Subdivisions ${\cal D}_1$, ${\cal D}_2$, and ${\cal D}_3$ of a benzenoid system from Figure \ref{fig:prim}.}
\end{figure}

	Let ${\cal C}_i$ be the set of all cut segments of type $i$. Furthermore, let ${\cal C}_i'$ be the set of all vertices of $Z$ that are not on any cut segment of type $i$. Now we define a new graph $\Gamma_i$ whose vertices are the elements in the set ${\cal C}_i \cup {\cal C}_i'$ and two vertices of $\Gamma_i$ are adjacent if and only if the corresponding elements belong to a common strip of ${\cal D}_i$. From the definition of $\Gamma_i$ it follows that it is a tree. Note that $\Gamma_i$ can be derived from ${\cal D}_i$ in linear time.

	An edge of $\Gamma_i$ is called {\em thick} if it is defined by a strip with value $1$ and {\em thin} otherwise. Every cut segment of ${\cal C}_i$ is incident to exactly one thick edge, all remaining vertices of $\Gamma_i$ being incident only to thin edges. Moreover, we define a weight for each thick edge $f$ as follows. Let ${\cal F}$ be the rectangular strip of ${\cal D}_i$ corresponding to $f$. Edge $f$ gets a weight equal to the number of edges in $G$ that lie completely in ${\cal F}$.

Finally, we define a weight for all the vertices of $\Gamma_i$ as follows. Any element of ${\cal C}_i'$ gets weight $1$ and any element of ${\cal C}_i$ gets the same weight as the thick edge incident to it. Obviously, the weight of any vertex in ${\cal C}_i$ represents the number of vertices of $G$ lying on the corresponding cut segment. If we contract all thin edges of $\Gamma_i$ and use the described weights also on the new tree, we obtain the vertex-edge weighted tree (see Figure \ref{fig:trees1}). It is obvious from the construction that the obtained tree is exactly $(T_i,w_i,w_i')$.

	\begin{figure}[h!] 
\begin{center}
\includegraphics[scale=1.2]{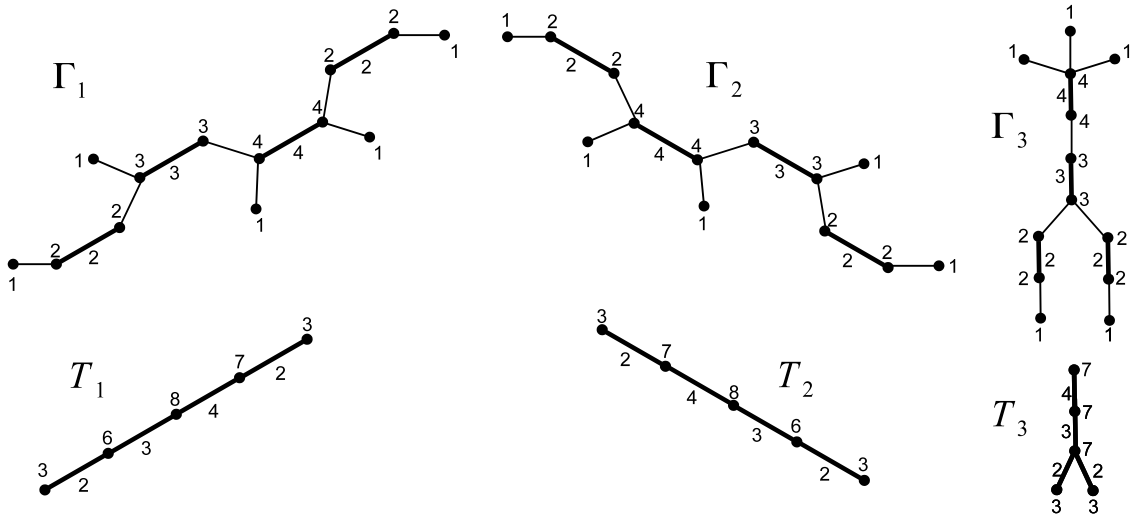}
\end{center}
\caption{\label{fig:trees1} Weighted trees $\Gamma_i$ and $T_i$, $i \in \lbrace 1,2,3 \rbrace$, of a benzenoid system from Figure \ref{fig:prim}.}
\end{figure}

	From this construction it follows that $(T_i,w_i,w_i')$ can be obtained in $O(|Z|)$ time.
\qed
\end{proof}

Next, we define the Szeged index of a vertex-edge weighted graph $(G,w,w')$ as
$$ Sz(G,w,w')  =  \sum_{e \in E(G)} w'(e)n_1(e|G)n_2(e|G), $$
where for $i \in \lbrace 1,2 \rbrace$ we have
$$n_i(e|G)  =  \sum_{x \in N_i(e|G)} w(x).$$

\noindent The following theorem is important for our consideration. 
\begin{theorem} \cite{chepoi-1997,NM-2017}
\label{szeged_izrek} 
If $G$ is a benzenoid system, then
$$ Sz(G)= Sz(T_1,w_1,w_1') + Sz(T_2, w_2,w_2') + Sz(T_3, w_3, w_3').$$
\end{theorem}

\noindent To compute the Szeged index of benzenoid systems in sub-linear time, we also need the next lemma.

\begin{lemma} \cite{chepoi-1997,NM-2017}
\label{lema_sz}
Let $(T,w,w')$ be a vertex-edge weighted tree with $n$ vertices. Then the Szeged index $Sz(T,w,w')$ can be computed in $O(n)$ time.
\end{lemma}

\noindent The main result of this section now follows easily.

\begin{theorem}
Let $G$ be a benzenoid system with the boundary cycle $Z$. Then the Szeged index $Sz(G)$ can be computed in $O(|Z|)$ time.
\end{theorem}
\begin{proof}
From Lemma \ref{tree_O(z)} it follows that the trees $(T_1,w_1,w_1')$, $(T_2, w_2,w_2')$, and $(T_3, w_3, w_3')$ can be obtained in $O(|Z|)$ time. By Lemma \ref{lema_sz}, the Szeged index of each mentioned tree can be computed in linear time with respect to $|Z|$. Finally, using  Theorem \ref{szeged_izrek}, the Szeged index $Sz(G)$ can be computed in $O(|Z|)$ time.
\qed
\end{proof}

\section{The edge-Wiener index, the edge-Szeged index, and the PI index}

In the present section we show that the edge-Wiener index, the edge-Szeged index, and the PI index of a benzenoid system $G$ can be computed in sub-linear time. As in previous section we construct weighted trees. However, the weights are not the same as before.

\bigskip

Let $T_1$, $T_2$, $T_3$ be the quotient trees defined in the preliminaries. In this section we extend the quotient trees to weighted trees $(T_i, w_i)$, $(T_i, w'_i)$, $(T_i, w_i, w'_i)$ as follows: 
\begin{itemize}
\item for $C \in V(T_i)$, let $w_i(C)$ be the number of edges in the component $C$ of $G_i = G  - E_i$;
\item for $E = C_1C_2 \in E(T_i)$, let $w_i'(E)$ be the number of edges between components $C_1$ and $C_2$.
\end{itemize}

In the next lemma we prove that each of the trees $(T_1, w_1, w_1')$, $(T_2, w_2, w_2')$, and $(T_3, w_3, w_3')$ can be obtained in sub-linear time with respect to the number of vertices in the benzenoid system. Note that the first part of the proof of the following lemma is the same as in the proof of Lemma \ref{tree_O(z)}. For this reason, we skip some steps of the proof. 
\begin{lemma}\label{T_O(z)}
	Let $G$ be a benzenoid system and let $Z$ be its boundary cycle. Then each tree $(T_i, w_i, w_i')$, $i \in \{1,2,3\}$, can be obtained in $O(|Z|)$ time.	
\end{lemma}
\begin{proof}
Let $i \in \{1,2,3\}$ and define the tree $\Gamma_i$ exactly as in the proof of Lemma \ref{tree_O(z)}. Now, we define a weight just for all the edges of the graph $\Gamma_i$. If $f$ is an edge of $\Gamma_i$, let ${\cal F}$ be a strip of ${\cal D}_i$ corresponding to $f$. Edge $f$ gets a weight equal to the number of edges of $G$ that lie completely in ${\cal F}$. Note that the thick edges are weighted as before.
	
Finally, we contract all thin edges of $\Gamma_i$ and obtain a weighted tree in the following way. Every edge of a new tree is weighted by the weight of the corresponding thick edge. The weight of any vertex of this new tree is defined as the sum of all the weights of thin edges incident to the corresponding vertex in $\Gamma_i$ (see Figure \ref{fig:trees2}). It is obvious from the construction that the obtained tree is exactly $(T_i,w_i,w_i')$.

	\begin{figure}[h!] 
\begin{center}
\includegraphics[scale=1.2]{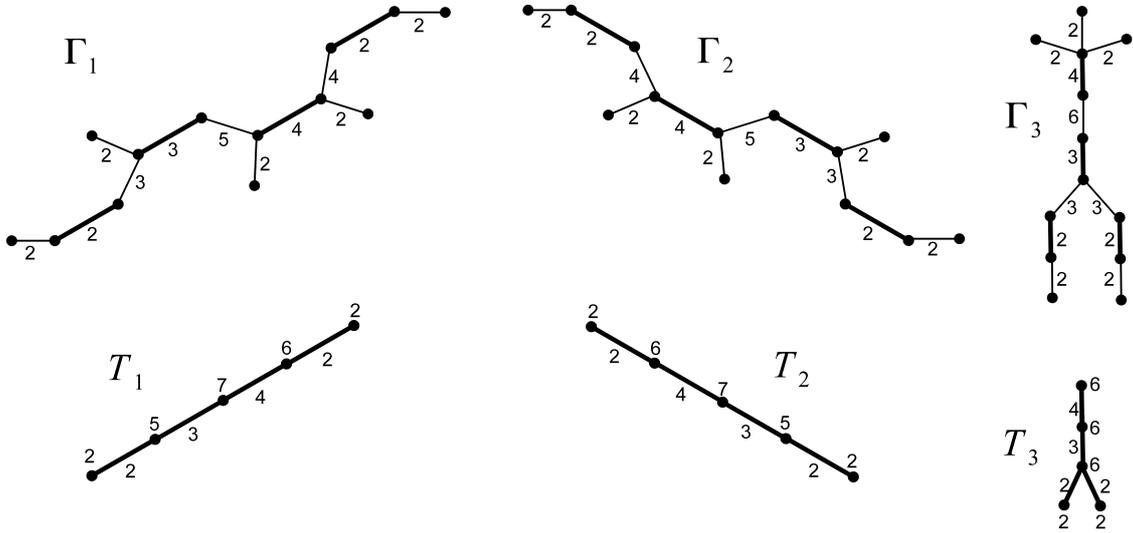}
\end{center}
\caption{\label{fig:trees2} Weighted trees $\Gamma_i$ and $T_i$, $i \in \lbrace 1,2,3 \rbrace$, of a benzenoid system from Figure \ref{fig:prim}.}
\end{figure}

From this construction it follows that $(T_i,w_i,w_i')$ can be obtained in $O(|Z|)$ time.
\qed
\end{proof}

\subsection{The edge-Wiener index}

In the preliminaries we defined the edge-Wiener index in such a way that the distance between two edges of a graph is equal to the distance between corresponding vertices in the line graph. On the other hand, for edges $e = ab$ and $f = xy$ of a graph $G$ it is also legitimate to set

$$\widehat{d}_G(e,f) = \min \lbrace d_G(a,x), d_G(a,y), d_G(b,x), d_G(b,y) \rbrace.$$

\noindent
Obviously, for any two distinct edges $e,f \in E(G)$ it holds $d_G(e,f) = \widehat{d}_G(e,f) + 1$. Replacing $d$ with $\widehat{d}$ in~\eqref{def_ew}, a variant of the edge-Wiener index from~\cite{khalifeh-2009} is obtained, let us denote it with $\widehat{W}_e(G)$. It is easy to see that $W_e(G)$ and $\widehat{W}_e(G)$ are connected in the following way (cf.~\cite[Corollary 8]{iranmanesh-2009} and~\cite[Theorem 2.4]{khalifeh-2009}):

$$\widehat{W}_e(G) =  W_e(G) - \binom{|E(G)|}{2}.$$

The definitions of the Wiener indices can be extended to weighted graphs as follows. Let $G$ be a graph and let $(G,w)$, $(G,w')$, and $(G,w,w')$ be a vertex-weighted graph, an edge-weighted graph, and a vertex-edge weighted graph, respectively. The corresponding Wiener indices of these weighted graphs are defined as
\begin{eqnarray*}
W(G,w) & = & \frac{1}{2} \sum_{x \in V(G)} \sum_{y \in V(G)} w(x)w(y)d_G(x,y)\,, \\
\widehat{W}_e(G,w') & = & \frac{1}{2} \sum_{e \in E(G)} \sum_{f \in E(G)} w'(e)w'(f)\widehat{d}_G(e,f)\,,\\
W_{ve}(G,w,w') & = & \sum_{x \in V(G)} \sum_{e \in E(G)} w(x)w'(e) d_G(x,e)\,. 
\end{eqnarray*}

\noindent The following result says that the edge-Wiener index of a benzenoid system can be computed as the sum of Wiener indices of weighted quotient trees.

\begin{theorem} \cite{kelenc}
\label{eW_izrek}
If $G$ is a benzenoid system, then 
$$\widehat{W}_e(G)= \sum_{i=1}^3 \left( \widehat{W}_e(T_i, w_i') + W(T_i, w_i) + W_{ve}(T_i, w_i, w_i')\right).$$
\end{theorem}

\noindent The following lemma claims that the Wiener indices of weighted quotient trees can be computed in linear time.

\begin{lemma} \cite{kelenc}
\label{eW_lin}
Let $(T,w,w')$ be a vertex-edge-weighted tree with $n$ vertices and $m$ edges. Then the indices $\widehat{W}_e(T, w')$, $W(T, w)$, and $W_{ve}(T, w, w')$ can be computed in $O(n) = O(m)$ time.
\end{lemma}

\begin{proof}
See the proof of Theorem 4.1 in \cite{kelenc}. \qed
\end{proof}

\noindent
Finally, we are able to proof the main result of this subsection.

\begin{theorem}
Let $G$ be a benzenoid system with the boundary cycle $Z$. Then the edge-Wiener index $W_e(G)$ can be computed in $O(|Z|)$ time.
\end{theorem}

\begin{proof}
From Lemma \ref{T_O(z)} it follows that the trees $(T_1,w_1,w_1')$, $(T_2, w_2,w_2')$, and $(T_3, w_3, w_3')$ can be obtained in $O(|Z|)$ time. By Lemma \ref{eW_lin}, the Wiener indices of each mentioned tree can be computed in linear time with respect to $|Z|$. Finally, using  Theorem \ref{eW_izrek}, the edge-Wiener index $W_e(G)$ can be computed in $O(|Z|)$ time.
\qed
\end{proof}

\subsection{The edge-Szeged index}

To efficiently calculate the edge-Szeged index of a benzenoid system, we define the \textit{total-Szeged index} of a vertex-edge weighted graph $(G,w,w')$ as
$$Sz_t(G,w,w') = \sum_{e \in E(G)}w'(e)r_1(e|G)r_2(e|G),$$ 
\noindent
where for $i \in \lbrace 1, 2 \rbrace$ we have
$$r_i(e|G) = \sum_{x \in N_i(e|G)} w(x) + \sum_{f \in M_i(e|G)} w'(f).$$

\noindent We can now state the following theorem to express the edge-Szeged index as the sum of total-Szeged indices of weighted quotient trees.

\begin{theorem}
\label{eS_izrek} \cite{tratnik}
If $G$ is a benzenoid system, then
$$ Sz_e(G)= Sz_t(T_1,w_1,w_1') + Sz_t(T_2, w_2,w_2') + Sz_t(T_3, w_3, w_3').$$
\end{theorem}

\noindent For the fast computation of the edge-Szeged index we also need the following lemma.

\begin{lemma}
\label{lema_edge_sz} \cite{tratnik}
Let $(T,w,w')$ be a vertex-edge weighted tree with $n$ vertices and $m$ edges. Then the total-Szeged index $Sz_t(T,w,w')$ can be computed in $O(m) = O(n)$ time.
\end{lemma}

\noindent The main result of this subsection now follows easily.

\begin{theorem}
Let $G$ be a benzenoid system with the boundary cycle $Z$. Then the edge-Szeged index $Sz_e(G)$ can be computed in $O(|Z|)$ time.
\end{theorem}

\begin{proof}
By Lemma \ref{T_O(z)}, the trees $(T_1,w_1,w_1')$, $(T_2, w_2,w_2')$, and $(T_3, w_3, w_3')$ can be obtained in $O(|Z|)$ time. From Lemma \ref{lema_edge_sz} it follows that the total-Szeged index of any mentioned tree can be computed in linear time with respect to $|Z|$. Finally, using  Theorem \ref{eS_izrek}, the edge-Szeged index $Sz_e(G)$ can be computed in $O(|Z|)$ time.
\qed
\end{proof}

\subsection{The PI index}

For the PI index we first define the edge-PI index and the vertex-PI index of an edge-weighted graph $(G,w')$ and vertex-edge weighted graph $(G,w,w')$, respectively, in the following way

\begin{eqnarray*}
PI_e(G,w') & = & \sum_{e \in E(G)} w'(e)\big(m_1(e|G) + m_2(e|G)\big),\\
PI_v(G,w,w') & = & \sum_{e \in E(G)} w'(e)\big(n_1(e|G) + n_2(e|G)\big),
\end{eqnarray*}

\noindent where for $i \in \lbrace 1, 2 \rbrace$ we have

\begin{eqnarray*}
n_i(e|G) & = & \sum_{x \in N_i(e|G)} w(x), \\
m_i(e|G) & = & \sum_{f \in M_i(e|G)} w'(f).
\end{eqnarray*}

\noindent The following theorem is crucial for the main result.

\begin{theorem}
\label{pi_sum} \cite{tratnik}
If $G$ is a benzenoid system, then
$$PI_e(G)= \sum_{i=1}^{3}\big( PI_e(T_i,w_i') + PI_v(T_i,w_i,w_i') \big).$$
\end{theorem}

\noindent For the fast computation of the PI index we also need the following two lemmas.

\begin{lemma}\cite{tratnik}
\label{pi_lema1} 
Let $(T,w')$ be an edge-weighted tree with $n$ vertices and $m$ edges. Then the PI index $PI_e(T,w')$ can be computed in $O(m) = O(n)$ time.
\end{lemma}

\begin{lemma}\cite{tratnik}
Let $(T,w,w')$ be a vertex-edge weighted tree with $n$ vertices and $m$ edges. Then the vertex-PI index $PI_v(T,w,w')$ can be computed in $O(m) = O(n)$ time.
\label{pi_lema2}
\end{lemma}

\noindent We are ready to show the final result of this subsection.

\begin{theorem}
Let $G$ be a benzenoid system with the boundary cycle $Z$. Then the PI index $PI(G)$ can be computed in $O(|Z|)$ time.
\end{theorem}

\begin{proof}
By Lemma \ref{T_O(z)}, the trees $(T_1,w_1,w_1')$, $(T_2, w_2,w_2')$, and $(T_3, w_3, w_3')$ can be obtained in $O(|Z|)$ time. From Lemma \ref{pi_lema1} and Lemma \ref{pi_lema2} it follows that the edge-PI index and the vertex-PI index of any mentioned tree can be computed in linear time with respect to $|Z|$. Finally, using  Theorem \ref{pi_sum}, the PI index $PI(G)$ can be computed in $O(|Z|)$ time.
\qed
\end{proof}

\section*{Acknowledgement}

The author Matev\v z \v Crepnjak acknowledge the financial support from the Slovenian Research Agency (research core funding No. P1-0285). 

\noindent The author Niko Tratnik was finacially supported by the Slovenian Research Agency.

\end{document}